\documentclass[11pt]{amsart}
\usepackage{amsmath}
\usepackage{amsfonts}
\usepackage{amssymb}
\usepackage{amsthm}
\usepackage{cite}
\usepackage{amscd}
\usepackage[margin=3.5cm]{geometry}

\newtheorem{thm}{Theorem}[section]
\newtheorem{cor}[thm]{Corollary}
\newtheorem{lem}[thm]{Lemma}
\newtheorem{prop}[thm]{Proposition}

\theoremstyle{definition}
\newtheorem{defn}{Definition}[section]

\theoremstyle{remark}
\newtheorem{rem}{Remark}[section]

\newcommand{\AAA}{\mathcal{A}}
\newcommand{\BB}{\mathcal{B}}
\newcommand{\CC}{\mathcal{C}}
\newcommand{\DD}{\mathcal{D}}
\newcommand{\EE}{\mathcal{E}}
\newcommand{\FF}{\mathcal{F}}
\newcommand{\GG}{\mathcal{G}}
\newcommand{\HH}{\mathcal{H}}
\newcommand{\II}{\mathcal{I}}
\newcommand{\JJ}{\mathcal{J}}
\newcommand{\KK}{\mathcal{K}}
\newcommand{\LL}{\mathcal{L}}
\newcommand{\MM}{\mathcal{M}}
\newcommand{\NN}{\mathcal{N}}
\newcommand{\OO}{\mathcal{O}}
\newcommand{\PP}{\mathcal{P}}

\newcommand{\SSS}{\mathcal{S}}
\newcommand{\TT}{\mathcal{T}}
\newcommand{\UU}{\mathcal{U}}

\newcommand{\XX}{\mathcal{X}}
\newcommand{\YY}{\mathcal{Y}}

\newcommand{\Pone}{\mathbb{P}^1}

\newcommand{\Flags}[1]{\textnormal{Fl}(#1)^s}
\newcommand{\cell}{\Omega^{\circ}}
\newcommand{\Hom}[5]{\textnormal{Hom}_{\mathcal{#1}}(#2,#3,\mathcal{#4},\mathcal{#5})}

\newcommand{\UnivInt}[1]{\textbf{U}_{\mathcal{#1}}}
\newcommand{\UnivHom}[2]{\textbf{H}_{\mathcal{#1},\mathcal{#2}}}
\newcommand{\UnivPairedInt}[2]{\textbf{U}_{\mathcal{#1},\mathcal{#2}}}
\newcommand{\UnivPairedHom}[3]{\textbf{H}_{\mathcal{#1},\mathcal{#2},\mathcal{#3}}}

\newcommand{\LR}{c_{\lambda\mu}^{\nu}}
\newcommand{\StretchedLR}{c_{N\lambda N\mu}^{N\nu}}
\newcommand{\SL}{\textnormal{SL}_r}

\begin{document}
\title{Geometric Proof of a Conjecture of King, Tollu, and Toumazet}
\author{Cass Sherman}
\address{Department of Mathematics, University of North Carolina at Chapel Hill, Phillips Hall, Chapel Hill, NC 27599}
\email{cas1987@email.unc.edu}
\begin{abstract}
King, Tollu, and Toumazet in  \cite{KTT} conjectured that stretching the parameters of a Littlewood-Richardson coefficient of value $2$ by a factor of $N$ would result in a coefficient of value $N+1$. We prove a slight generalization of this by using geometric methods and Schubert calculus.
\end{abstract}
\maketitle

\setcounter{section}{-1}

\begin{section}{Introduction}\label{Intro}
Given Young diagrams $\lambda$, $\mu$, $\nu$ with at most $r$ rows, the associated Littlewood-Richardson number $\LR$ computes the dimension of the space of $\SL$ invariants of the tensor product $V_\lambda\otimes V_\mu\otimes V_{\nu}^*$, where as usual $V_\lambda$ denotes the irreducible polynomial representation of $\textnormal{GL}_r$ corresponding to $\lambda$. Given a whole number $N$, each row of the Young diagrams $\lambda$, $\mu$, $\nu$ can be stretched by a factor of $N$ (so e.g. if $N=2$, each row becomes twice as long, etc) and one may ask how does the number $P(N)=\StretchedLR$ change with $N$? Fulton conjectured (unpublished) and Knutson, Tao, and Woodward later proved \cite{KTW2} that if $P(1)=1$, then $P(N)=1$ for all $N$. This fact is related to irredundancy of a certain set of inequalities appearing in Horn's conjecture \cite{Bel04}.

A natural next question would be what if $P(1)=2$? It was given a correct conjectural answer of $P(N)=N+1$ by King, Tollu, and Toumazet in \cite{KTT} and was proven by Ikenmeyer in \cite{Ikenmeyer}. Ikenmeyer interprets $\LR$ as the cardinality of the set of integral hive flows on the honeycomb graph of $r$ with borders prescribed by $\lambda$, $\mu$, $\nu$. He then uses combinatorial and algorithmic techniques to count the hive flows and arrives at the conjectured answer. We will not pursue his methods, although he thinks that they should generalize, e.g. to the case $P(1)=3$. 

Instead we will prove the conjecture using the geometric methods established by Belkale \cite{Bel06,Bel07}. For this, it is helpful to recast the question as follows. Let $s\geq3$, $n>r$ be integers, and let $\lambda^1$,..., $\lambda^s$ be Young diagrams fitting in an $r$ by $n-r$ box (equivalently, weights of $G:=\SL$ of level $n-r$). Then,

\begin{thm}\label{MainTheorem}
Suppose $\sum_{p=1}^s|\lambda^p|=r(n-r)$ (the ``codimension condition''). If
\begin{equation*}
\dim(V_{\lambda^1}\otimes...\otimes V_{\lambda^s})^G=2
\end{equation*}
then
\begin{equation*}
\dim(V_{N\lambda^1}\otimes...\otimes V_{N\lambda^s})^G=N+1
\end{equation*}
for all integers $N\geq1$
\end{thm}
\begin{rem}
This generalizes the conjecture of \cite{KTT} proven in \cite{Ikenmeyer} to an arbitrary number of weights. Indeed, suppose $\LR=2$. Then in particular $|\lambda|+|\mu|=|\nu|$, for this is the case whenever $\LR\neq0$. Choose $n$ large enough that $\lambda$, $\mu$, $\nu$ each fit in an $r\times (n-r)$ box. Let $\nu^{\vee}$ be the Young diagram with $(\nu^{\vee})_a=n-r-\nu_{r-a+1}$. One verifies that $V_{N\nu^{\vee}} \cong V_{N\nu}^*$ as $\textnormal{GL}_r$ representations for any $N\geq1$, and that $|\lambda|+|\mu|+|\nu^\vee|=r(n-r)$. Thus, by Theorem \ref{MainTheorem} with $s=3$, we have $N+1=(V_{N\lambda}\otimes V_{N\mu}\otimes V_{N\nu^\vee})^G=\StretchedLR$. $\square$
\end{rem}

To prove Theorem \ref{MainTheorem}, we further translate the question into one of the size of a moduli space $\MM$, which can be described as the $s$-fold product of the space of complete flags on an $r$-dimensional vector space $V$ modulo an equivalence relation - the theorem holds if and only if $\MM$ is 1-dimensional (see Section \ref{ConjectureTranslated}). This is the technique used by Belkale \cite{Bel07} to prove Fulton's conjecture. The moduli space under consideration has an ample line bundle $\LL$ with the property that the global sections of $\LL^{\otimes N}$ can be identified with the $G$ invariants of $V_{N\lambda^1}\otimes...\otimes V_{N\lambda^s}$ (really its dual). The idea is that if $\MM$ has dimension exceeding 1, then $\LL$ will have a nonempty base locus $Z$. From the properties of $\LL$, if $\FF\in\MM$ lies in $Z$, then a certain vector space $H$ depending on $\FF$ will be nonzero. However, this $H$ must be 0 whenever $\FF$ consists of ``general'' flags. While $\FF$ itself cannot be taken to be general (for it must lie in $Z$), we can ``trade'' $\FF$ for the flags induced by $\FF$ on a subspace of $V$. These flags will be general enough to allow us to conclude that $H$ must have been 0, a contradiction. That is, $\MM$ must have had dimension 1 all along. The ``trading'' process derives from techniques of Schofield \cite{Schofield} (see e.g. his Theorem 5.2), modified so that his Hom and Ext become our $H^0$ and $H^1$ of certain two-step complexes.

The author wishes to acknowledge many useful discussions with his thesis advisor P. Belkale. In particular, I would like to thank him for pointing out the technique of Schofield in connection to this problem and for showing me how to correct a significant error in earlier versions of this paper.

\begin{subsection}{Notation}\label{Notation}
Throughout $k$ will be an algebraically closed field of characteristic 0. The term ``vector space'' should be understood to mean finite dimensional over $k$. In particular, $V$ will have dimension $r$, $M$ will have dimension $m$, $Q$ will have dimension $n-r$, and $W\cong M\oplus Q$ will of course have dimension $n-r+m$. 

The quantity $[n]$ for a positive integer $n$ will denote the set of integers $\{1,...,n\}$. Italicized capital letters will denote \textit{index sets}. An \textit{index set} in $[n]$ is a subset of $[n]$ which is written in ascending order. If the index set $I$ in $[n]$ has cardinality $r$, we will say that $I\in\binom{[n]}{r}$. It is typical to associate to an index set $I\in\binom{[n]}{r}$ a Young diagram $\lambda(I)$, whose $a$th row is given by the equation $\lambda_a=n-r+a-I_a$. Finally, the lowercase letter $s$ will always denote a fixed positive integer greater than or equal to $3$.
\end{subsection}

\end{section}

\begin{section}{Preliminaries}\label{Preliminaries}

\begin{subsection}{Schubert Calculus}\label{Schubert Calculus}
The Grassmannian $\textnormal{Gr}(m,W)$ of $m$-dimensional subspaces of $W$ is a smooth, projective variety of dimension $m(n-r)$. It has distinguished subvarieties called Schubert varieties, each of which corresponds to a choice of full flag $E_{\bullet}$ on $W$ and a choice of index set $H\in\binom{[n-r+m]}{m}$. Explicitly, for such an $E_{\bullet}$ and $H$, the Schubert variety is defined as
\begin{equation*}
\Omega_H(E_{\bullet})=\{M\in\textnormal{Gr}(m,W)|\dim(M\cap E_a)\geq b \textnormal{ whenever }H_b\leq a<H_{b+1},\textnormal{ }b=1,...,m\}.
\end{equation*}
It has codimension $\sum_{a=1}^{r}n-r+a-H_a$. Each Schubert variety has a distinguished Zariski open subset, isomorphic to affine space, called the Schubert cell. It is defined as
\begin{equation*}
\cell_H(E_{\bullet})=\{M\in\textnormal{Gr}(m,W)|\dim(M\cap E_a)=b \textnormal{ iff } H_b\leq a < H_{b+1},\textnormal{ }b=1,...,m\}
\end{equation*}
For fixed $E_{\bullet}$, the Schubert cells over all choices of $H$ disjointly cover the Grassmannian. As a result of this cell decomposition, the classes $\omega_H$ of the Schubert varieties form an additive basis for the integral cohomology ring $H^*(\textnormal{Gr}(m,W),\mathbb{Z})$.

We will be interested mainly in the intersections of $s$-many Schubert cells. Let $\EE \in \Flags{W}$ have entries $E^p_{\bullet}$ for $p=1,...,s$. Given an $s$-tuple of index sets $\HH \in \binom{[n-r+m]}{m}^s$, we make the notational convention
\begin{equation*}
\cell_\HH(\EE):=\bigcap_{p=1}^s\cell_{H^p}(E^j_{\bullet})
\end{equation*}
and correspondingly
\begin{equation*}
\omega_{\HH}:=\prod_{p=1}^s\omega_{H^p}.
\end{equation*}
The latter is nonzero if and only if the former is nonempty for general choice of $\EE$ in $\Flags{W}$. If $M\in\cell_\HH(\EE)$, we will say that $M$ is \textit{in Schubert position} $\HH$ \textit{with respect to} $\EE$, regardless of genericity of $\EE$.

One may also detect nonzeroness of a Schubert product in another way. The idea comes from Kleiman transversality. The tangent space to the Schubert intersection $\cell_\HH(\EE)$ at a point $M$ is canonically identified with the vector space:
\begin{equation*} 
\{\phi\in\textnormal{Hom}(M,W/M)|\phi(\EE(M)^p_a)\subseteq\EE(W/M)^p_{H^p_a-a}\textnormal{ for }a=1,...,m,\textnormal{ }p=1,...,s\},
\end{equation*}
where $\EE(M)$ (resp. $\EE(W/M)$) indicates the $s$-tuple of flags induced on $M$ (resp. $W/M$) by $\EE$. If the product of Schubert classes is nonzero, then for sufficiently general $\FF$ and $M$ in the intersection $\cell_\HH(\EE)$, the intersection is transverse at $M$. That is, the tangent space at $M$ has its expected dimension $m(n-r)-\sum_{p=1}^s\sum_{a=1}^m (n-r+a-H^p_a)$.

A strong converse to this is also true. To state it, we define a generalized tangent space. For any $m$-dimensional $M$ (not a priori a subspace of $W$) and $n-r$ dimensional $Q$, with $s$-tuples of flags $\FF\in\Flags{M}$ and $\GG\in\Flags{Q}$, we define
\begin{equation*}
\Hom{H}{M}{Q}{F}{G}=\{\phi\in\textnormal{Hom}(M,Q)|\phi(F^p_a)\subseteq G^p_{H^p_a-a}\textnormal{ for }a=1,...,m,\textnormal{ }p=1,...,s\}.
\end{equation*}
The next proposition is Proposition 2.3 in \cite{Bel06}.

\begin{prop}\label{TangentCondition}
$\omega_\HH\neq0$ if and only if for general $(\FF,\GG)\in\Flags{M}\times\Flags{Q}$, one has
\begin{equation}
\dim\Hom{H}{M}{Q}{F}{G}=m(n-r)-\sum_{p=1}^s\sum_{a=1}^m (n-r+a-H^p_a).
\end{equation}
\end{prop}
\end{subsection}

\begin{subsection}{Parabolic Vector Spaces}\label{Parabolic}
A \textit{parabolic vector space} is a 3-tuple $(M,\FF,\lambda)$ consisting of a vector space $M$ of dimension $m$, an element $\FF$ of $\Flags{M}$, and an $s$-tuple $\lambda$ of nonincreasing sequences of real numbers with each sequence of having length $m$. Given an $e$-dimensional subspace $R$ of $M$ whose Schubert position in $M$ with respect to $\FF$ is given by the $s$-tuple of index sets $\EE\in\binom{[m]}{e}$, we define the \textit{parabolic slope} of $R$ to be:
\begin{equation}\label{Slope}
\mu_R=(1/e)\sum_{p=1}^s\sum_{a\in E^p}\lambda^p_a
\end{equation}
(note the unfortunate change of the meaning of $\EE$ from Section \ref{Schubert Calculus}). A parabolic vector space is said to be \textit{semistable} if for every subspace $R$ of $M$, one has $\mu_R\leq\mu_M$. 

Given $M$ and $\FF$ as above, an integer $n-r>0$ and an $s$-tuple of index sets $\HH\in\binom{[n-r+m]}{m}^s$, we may associate a parabolic vector space
\begin{equation*}
(M,\FF,n-r,\lambda(\HH))\textnormal{ where }\lambda(\HH)^p_a=n-r+a-H^p_a.
\end{equation*}  
Since the parabolic slope of a subspace depends only on its Schubert position, we define the slope $\mu_\EE$ of a Schubert position $\EE$ by the same formula (\ref{Slope}).
\end{subsection}

\begin{subsection}{Parameter Spaces}\label{Parameter Spaces}
The parameter spaces below will facilitate the key dimension calculations in the proof of Theorem \ref{MainTheorem}. For $\EE$ in $\binom{[m]}{e}^s$, as in \cite{Bel06} we introduce the ``universal intersection'' $\UnivInt{E}(M)$ whose points are pairs $(R,\FF)$, where $R$ is an $e$-dimensional subspace of $M$ in Schubert position $\EE$ with respect to $\FF\in\Flags{M}$. Also, we have a parameter space lying over $\UnivInt{E}$, denoted $\UnivHom{H}{E}(M,Q)$, whose fiber over $(R,F)$ is the set of pairs $(\GG,\phi)$, where $\GG\in\Flags{Q}$ and $\phi\in\Hom{H}{M}{Q}{F}{G}$ is such that $\textnormal{ker}\phi=R$. The vector spaces $M$ and $Q$ will often be omitted from the notation.

Let $V$ have dimension $r$. Fix integers $0\leq g<f<r$. Let
\begin{equation}\label{IndexSets}
\II\in\binom{[n]}{r}^s, \KK\in\binom{[r]}{f}^s, \JJ\in\binom{\mathcal{K}}{g}^s, \LL\in\binom{[n]}{f}^s,\NN\in\binom{[f]}{g}^s,
\end{equation}
where the first three are chosen arbitrarily, the fourth is given by $L^p_a=I^p_{K^p_a}$ for $a=1,...,f$, and the fifth is determined by the rule $J^p_a=K^p_{N^p_a}$ for $a=1,...,g$. Let $A_{f,f,g}(V)$ be the scheme over $\textnormal{Spec}(k)$ whose closed points are triples of the form $(S,S',T)$ where $S$ and $S'$ are $f$-dimensional subspaces of $V$ that intersect in a $g$-dimensional space $T$. Let $\UnivPairedInt{K}{J}(V)$ be the scheme over $A_{f,f,g}$ whose fiber over $(S,S',T)$ is the set of all $\FF\in\Flags{V}$ such that $S,S'\in\cell_{\KK}(\FF)$ and $T\in\cell_{\JJ}(\FF)$.

Also, let $\UnivPairedHom{I}{K}{J}(V,Q)$ be the scheme over $\UnivPairedInt{K}{J}$ whose fiber over a point $(S,S',T,\FF)$ is the set of quadruples $(\GG,\GG',\phi,\phi')$ where $\GG,\GG'\in\Flags{Q}$, $\phi\in\Hom{I}{V}{Q}{F}{G}$,  $\phi'\in\Hom{I}{V}{Q}{F}{G'}$, and $\phi,\phi'$ are such that $\textnormal{ker}\phi=S$, $\textnormal{ker}\phi'=S'$. The properties of these schemes, including their existence, is proven in Appendix \ref{Representability}. For convenience, we summarize the results below. 

\begin{lem}\label{ModuliDimension}
We will use the phrase ``$X$ is irreducible over $Y$'' for an irreducible scheme $Y$ to mean that for all irreducible schemes $Z$ over $Y$, the scheme $X\times_Y Z$ is irreducible. With this terminology, we have:
\begin{enumerate}
\item $\textnormal{\textbf{U}}_{\EE}(M)$ is surjective, smooth, and irreducible over $\textnormal{Gr}(e,M)$ of relative dimension $\dim\Flags{M}-\sum_{p=1}^s|\lambda(E^p)|$.
\item $\textnormal{\textbf{H}}_{\HH,\EE}(M,Q)$ is surjective, smooth, and irreducible over $\textnormal{\textbf{U}}_\EE(M)$ of relative dimension \[(m-e)(n-r)+\dim\Flags{Q}+\sum_{p\in S}\sum_{a=1}^{e}(n-r+E^p_a-H^p_{E^p_a})-\sum_{p\in S}|\lambda(H^p)|.\]
\item $A_{f,f,g}$ is irreducible and smooth over $\textnormal{Spec}(k)$ of dimension $2(f-g)(r-f)+g(r-g)$. 
\item $\textnormal{\textbf{U}}_{\KK,\JJ}(V)$ is surjective, smooth, and irreducible over $A_{f,f,g}$ of relative dimension
\[\dim\Flags{V}+\sum_{p=1}^s|\lambda(J^p)|-2\sum_{p=1}^s|\lambda(K^p)|-2\sum_{p=1}^s|\lambda(N^p)|.\]
\item $\textnormal{\textbf{H}}_{\II,\KK,\JJ}(V,Q)$ is surjective, smooth, and irreducible over $\textnormal{\textbf{U}}_{\KK,\JJ}(V)$ of relative dimension
\[2(r-f)(n-r)+\dim\textnormal{Fl}(Q)^{2s}+2\sum_{p=1}^{s}(|\lambda(L^p)|-|\lambda(I^p)|-|\lambda(K^p)|).\]
\end{enumerate}
\end{lem}
\end{subsection}

\begin{subsection}{Hom Data}\label{Hom Data}
We define ``generic'' configurations for morphisms $M\rightarrow Q$ of vector spaces equipped with $s$-tuples of complete flags. The first such configuration will be called the $\textnormal{Hom}_{\HH}$ data. It is the set theoretic function $\textnormal{hd}_{\HH}$ on $\Flags{M}\times\Flags{Q}$ which assigns to a pair $(\FF,\GG)$ the triple $(D,e,\EE)$, where $D$ is the rank of $\Hom{H}{M}{Q}{F}{G}$, $e$ is the minimum of the dimensions of $\ker\phi$ as $\phi$ ranges over $\Hom{H}{M}{Q}{F}{G}$, and $\EE$ is the Schubert position of $\ker\phi$ in $M$ with respect to $\FF$ for all $\phi$ in some dense open subset of $\Hom{H}{M}{Q}{F}{G}$. Such an $\EE$ exists, as one sees by stratifying $\Hom{H}{M}{Q}{F}{G}$ into disjoint subschemes indexed by the Schubert position of the kernel. These subschemes are constructible, cover $\Hom{H}{M}{Q}{F}{G}$, and there are only finitely many of them. Hence exactly one contains an open set.

Similarly, we define the $\textnormal{Hom}'_{\HH}$ data to be the set theoretic function $\textnormal{hd}'_\HH$ on $\Flags{M}\times\Flags{Q}\times\Flags{Q}$ which assigns to $(\FF,\GG_1,\GG_2)$ the octuple $(D^1,D^2,e^1,e^2,t,\EE^1,\EE^2,\TT)$ where
\begin{itemize}
\item$D^i=\dim\textnormal{Hom}_{\HH}(M,Q,\FF,\GG_i)$.
\item$e^i=\min(\dim(\ker\phi_i))$ where the min is taken over $\phi_i$ in $\textnormal{Hom}_{\HH}(M,Q,\FF,\GG_i)$.
\item$t=\min(\dim(\ker\phi_1\cap\ker\phi_2))$ where the min is taken over the open subset of pairs $(\phi_1,\phi_2)$ in $\textnormal{Hom}_{\HH}(M,Q,\FF,\GG_1)\times \textnormal{Hom}_{\HH}(M,Q,\FF,\GG_2)$ such that $\dim(\ker\phi_i)=e^i$ for $i=1,2$. 
\item$\EE^i$ is the unique element of $\binom{[m]}{e^i}^s$ such that a nonempty open subset of $\phi_i$ in $\textnormal{Hom}_{\HH}(M,Q,\FF,\GG_i)$ satisfies $\ker\phi_i\in\cell_{\EE^i}(\FF)\subseteq\textnormal{Gr}(e^i,M)$.
\item$\TT$ is the unique element in $\binom{[m]}{t}^s$ such that a nonempty open subset of $(\phi_1,\phi_2)$ in $\textnormal{Hom}_{\HH}(M,Q,\FF,\GG_1)\times \textnormal{Hom}_{\HH}(M,Q,\FF,\GG_2)$ satisfies $\ker\phi_1\cap\ker\phi_2\in\cell_\TT(\FF)\subseteq\textnormal{Gr}(t,M)$.
\end{itemize}

\begin{defn}\label{GeneralElement}
Let $(\FF,\GG)\in\Flags{M}\times\Flags{Q}$ and suppose that $\textnormal{hd}_\HH(\FF,\GG)=(D,e,\EE)$. We will say that $\phi$ in $\Hom{H}{M}{Q}{F}{G}$ is a \textit{general element} if $\dim\ker\phi=e$ and $\ker\phi\in\cell_\EE(\FF)$. Similarly, let $(\FF,\GG_1,\GG_2)\in\Flags{M}\times\Flags{Q}\times\Flags{Q}$ and suppose that
\[\textnormal{hd}'_\HH(\FF,\GG_1,\GG_2)=(D^1,D^2,e^1,e^2,t,\EE^1,\EE^2,\TT).\]
We will say that $(\phi_1,\phi_2)$ in $\textnormal{Hom}_\HH(M,Q,\FF,\GG_1)\times\textnormal{Hom}_\HH(M,Q,\FF,\GG_2)$ is a \textit{general element} if $\dim(\ker\phi_i)=e^i$, $\dim(\ker\phi_1\cap\ker\phi_2)=t$, $\ker\phi_i\in\cell_{\EE^i}(\FF)$, $\ker\phi_1\cap\ker\phi_2\in\cell_\TT(\FF)$. In both cases, the set of general elements is nonempty and open.
\end{defn}

Take a closed subvariety $Z$ of $\Flags{M}$ and closed subvarieties $Y_1,Y_2$ of $\Flags{Q}$. The following lemma says that there is a generic $\textnormal{Hom}'_\HH$ data over $Z\times Y_1\times Y_2$. It is easy to verify using the fact that there are only finitely many possibilities for $\textnormal{Hom}'_{\HH}$ data ($D^i$ must be one of $0,1,...,m(n-r)$, $e^i$ must be one of $0,1,...,m$, etc.).

\begin{lem}\label{GeneralConfiguration}
There is a dense open set $U\subseteq Z\times Y_1\times Y_2$ such that $\textnormal{hd}'_\HH$ is constant over $(\FF,\GG_1,\GG_2)$ in $U$. Moreover, if $Y_1=Y_2$, one has $D^1=D^2$, $e^1=e^2$, $\EE^1=\EE^2$, and in this case $\textnormal{hd}_\HH=(D^1,e^1,\EE^1)$ for all $(\FF,\GG)$ in the image of either projection of $Z\times Y_1\times Y_1$ to $Z\times Y_1$.
\end{lem}
\end{subsection}

\end{section}

\begin{section}{GIT}\label{GIT}

This section is devoted to translating Theorem $\ref{MainTheorem}$ into a question of geometry, specifically of the dimension of a certain moduli space. The expert reader may wish to skip ahead to Theorem \ref{Descent}. The steps leading up to the theorem are standard, but the author has not seen them assembled to his satisfaction elsewhere, thus their inclusion below. 

\begin{subsection}{Borel-Weil Theory for SL}\label{Borel-Weil}
Suppose given an $s$-tuple of dominant weights $\lambda^1,...,\lambda^s$ for $\SL=\textnormal{SL}(V)$. View $\lambda^p$ as a Young diagram with at most $r-1$-many nonzero rows and suppose the distinct column lengths of $\lambda^p$ are 
\[r>d_1^p>d_2^p>...>d^p_{C(\lambda^p)}>0.\]
Let $b^p_i$ be the number of columns of $\lambda^p$ of length $d^p_i$. Finally, let $X^p$ be the partial flag variety consisting of flags 
\[k^r\supset F_{d_1^p}\supset F_{d_2^p}\supset...\supset F_{d_{C(\lambda^p)}^p}\supset0,\]
where subscripts denote dimension. One has a sequence of $\SL$-equivariant embeddings. 
\begin{multline*}
X^p\rightarrow\prod_{i=1}^{C(\lambda^p)}\textnormal{Gr}(d_i^p,r)\xrightarrow{Pl\ddot{u}cker}\prod_{i=1}^{C(\lambda^p)}\mathbb{P}(\wedge^{d_i^p}k^r)\xrightarrow{Veronese}\prod_{i=1}^{C(\lambda^p)}\mathbb{P}(\textnormal{Sym}^{b_i^p}(\wedge^{d_i^p}k^r))\\\xrightarrow{Segre}\mathbb{P}(\otimes_{i=1}^{C(\lambda^p)}\textnormal{Sym}^{b_i^p}(\wedge^{d_i^p}k^r)):=\mathbb{P}(\lambda^p).
\end{multline*}
Let $\LL^p$ denote the pullback of $\OO_{\mathbb{P}(\lambda^p)}(1)$ to $X^p$. Then $\LL^p$ is an $\SL$-equivariant line bundle on $X^p$. The quotient
\begin{equation*}
H^0(\mathbb{P}(\lambda^p),\OO_{\mathbb{P}(\lambda^p)}(1))\twoheadrightarrow H^0(X^p,\LL^p)
\end{equation*}
is isomorphic in the category of $\SL$ representations to the quotient
\begin{equation*}
(\otimes_{i=1}^{C(\lambda^p)}\textnormal{Sym}^{b_i^p}(\wedge^{d_i^p}k^r))^*\twoheadrightarrow V_{\lambda^p}^*.
\end{equation*}
See Chapter 9 of \cite{YoungTableaux} for full details.

If we instead start with $N\lambda^p$, the $d^p_i$ do not change, while the $b^p_i$ are multiplied by $N$. Let $W^p_i=\wedge^{d^p_i}k^r$ and note that the map $\textnormal{Pic}(\mathbb{P}(\textnormal{Sym}^{Nb_i^p}W_i^p))\rightarrow \textnormal{Pic}(\mathbb{P}(W_i^p))$ induced by the Veronese embedding is multiplication by $Nb_i^p$ (when both sides are identified with $\mathbb{Z}$ by $\OO(1)$). Since the pullback of $\OO(1)$ under the Segre embedding is the box tensor product of $\OO(1)$ on the factors, it follows that the pullback of $\OO_{\mathbb{P}(N\lambda^p)}(1)$ to $X^p$ is $(\LL^p)^{\otimes N}$. Therefore, $H^0(X^p,(\LL^p)^{\otimes N})\cong V^*_{N\lambda^p}$ as representations.

Let $X:=\prod_{p=1}^sX^p$, and define the $\SL$-equivariant line bundle $\tilde{\LL}_\lambda:=\boxtimes_{p=1}^s\LL^p$ on $X$. It now follows from the K\"{u}nneth formula that the space of sections of $\tilde{\LL}_\lambda^{\otimes N}$  is isomorphic as a representation to $\otimes_{p=1}^sV^*_{N\lambda^p}$. Note that $\tilde{\LL}_\lambda$ is the pullback of the box tensor product of $\OO(1)$'s under the embedding of $X$ into $\prod_{p=1}^s\mathbb{P}(\lambda^p)$, so in particular is very ample. To summarize,

\begin{prop}\label{LineBundle}
Let $\lambda^1,...,\lambda^s$ be dominant weights for $\SL$. There exists a product $X$ of $s$-many partial flag varieties and a very ample line bundle $\tilde{\LL}_{\lambda}$ on $X$ such that $H^0(\tilde\LL_{\lambda}^{\otimes N})$ is isomorphic as an $\SL$ representation to the tensor product of the representations $V_{N\lambda^p}^*$.
\end{prop}

\end{subsection}

\begin{subsection}{The Conjecture Translated into GIT}\label{ConjectureTranslated}
In the language of GIT, the $\SL$-equivariant line bundle $\tilde{\LL}_{\lambda}$ on $X$ is a linearization of the $\SL$ action on $X$. Given a point $x=(F^1_\bullet,...,F^s_\bullet)$ of $X$, the Hilbert-Mumford criterion equates the semistability of $x$ in the sense of GIT to the validity of a system of inequalities involving the partial flags $F^p_\bullet$ (see \cite{EigenvaluesOfSums}). These inequalities turn out to be precisely the ones defining parabolic semistability. More accurately, the point $x$ is semistable if and only if some (equivalently, any) point $\FF$ in the fiber over $x$ in $\Flags{V}$ is parabolic semistable, where stability is defined by the weights $\lambda^p$. Let $(\Flags{V})^{SS}$ denote the parabolic semistable locus on $\Flags{V}$. The proposition below now follows from standard GIT - see e.g. \cite{Newstead}, \cite{GeoInvTheory}.

\begin{prop}\label{GITProp}
There exists a projective, normal quotient $\MM$ of $X$ by the action of $\SL$, and a surjective morphism $\pi:(\Flags{V})^{SS}\rightarrow\MM$.
\end{prop}

Now let $\II\in\binom{[n]}{r}^s$ be such that $\lambda(\II)$ satisfies the codimension condition of Theorem \ref{MainTheorem}. Let $\tilde{\lambda}(\II)$ denote the $s$-tuple of Young diagrams obtained from $\lambda(\II)$ by truncating each row of $\lambda(\II)^p$ by the amount $\lambda(\II)^p_r$; the dominant weights of $\SL$ corresponding to $\lambda(\II)$ and $\tilde{\lambda}(\II)$ are the same. Defining $X$ and $\tilde{\LL}_{\II}:=\tilde{\LL}_{\tilde{\lambda}(\II)}$ as in Section \ref{Borel-Weil}, we obtain again a moduli space $\MM_\II$. In addition, the hypothesis on $\II$ ensures by the theory of Kempf \cite{PaulyEspaces} that the line bundle $\tilde{\LL}_\II$ \textit{descends} to $\MM_\II$.

\begin{thm}\label{Descent}
Let $\II\in\binom{[n]}{r}^s$ be such that
\begin{equation}\label{CodimCondition}
\sum_{p=1}^s\sum_{a=1}^r(n-r+a-I^p_a)=r(n-r).
\end{equation}
Then there exists $\MM_{\II}$ and $\pi$ as in Proposition \ref{GITProp} and an ample line-bundle $\LL_\II$ on $\MM_{\II}$ such that $\pi^*\LL_\II=\tilde\LL_\II|_{(\Flags{V})^{SS}}$. Moreover, the pullback
\begin{equation*}
\pi^*:H^0(\MM_{\II},\LL_{\II}^{\otimes N})\rightarrow H^0((\Flags{V})^{SS},\tilde{\LL}_{\II}^{\otimes N}|_{(\Flags{V})^{SS}})
\end{equation*}
has image given by the subspace $H^0(X,\tilde{\LL}_{\II}^{\otimes N})^{G}$ of $H^0(\Flags{V},\tilde{\LL}^{\otimes N}_{\II}|_{\Flags{V}})$ (note: invariant sections on the semistable locus extend uniquely to global invariant sections, see Lemma 4.15 of \cite{NarasimhanRamadas}). That is, by Proposition \ref{LineBundle},
\[H^0(\MM_{\II},\LL_\II^{\otimes N})\cong(V_{N\lambda^1}^*\otimes...\otimes V_{N\lambda^s}^*)^{\SL}.\]
\end{thm}

\begin{cor}\label{Translation}
Theorem \ref{MainTheorem} is equivalent to the following statement. If $\II$, $\MM_{\II}$, and $\LL_\II$ are as in Theorem \ref{Descent}, and $h^0(\MM_{\II},\LL_\II)=2$, then $h^0(\MM_{\II},\LL_\II^{\otimes N})=N+1$ for $N\geq1$.
\end{cor}

The degree of the Hilbert polynomial of an ample line bundle on a projective variety is the dimension of the variety, so by Corollary \ref{Translation}, $\MM_{\II}$ has dimension 1 if Theorem \ref{MainTheorem} holds. Conversely, suppose $\dim\MM_{\II}=1$ and $h^0(\MM_{\II},\LL_\II)=2$. Then, since $\MM_{\II}$ is dominated by the rational variety $(\Flags{V})^{SS}$, $\MM_{\II}$ itself is rational by L\"{u}roth's theorem. Therefore, $\MM_{\II}\cong\Pone$ and $\LL_\II$ must in turn be $\OO(1)$ (note: we've used here the normality of $\MM_{\II}$). Clearly, $h^0(\Pone,\OO(N))=N+1$, so by Corollary \ref{Translation}, Theorem \ref{MainTheorem} holds. Thus, it suffices to prove the following theorem.

\begin{thm}\label{MainTheoremTranslated}
(Equivalent to Theorem \ref{MainTheorem}) Suppose $\II\in\binom{[n]}{r}^s$ is such that (\ref{CodimCondition}) holds and that $h^0(\MM_{\II},\LL_\II)=2$. Then $\MM_{\II}$ is $1$-dimensional. In this case, the argument above gives $(\MM_{\II},\LL_\II)\cong(\Pone,\OO(1))$.
\end{thm}

\end{subsection}

\begin{subsection}{Theta Sections of $\LL_\II$}\label{ThetaSection}
Given $\GG\in\Flags{Q}$, Belkale \cite{Bel04_2} constructs a section $\theta(Q,\GG)$ in $H^0(\MM_{\II},\LL_\II)$. In his construction, there is an open dense subset of $(\Flags{Q})^{h^0(\MM_{\II},\LL_\II)}$ such that for $\GG_1,...,\GG_{h^0}$ in this subset, the set $\{\theta(Q,\GG_1),...,\theta(Q,\GG_{h^0})\}$ gives a basis of $H^0(\MM_{\II},\LL_\II)$. For our purposes, the only other important property of these sections is the vanishing loci of their pullbacks to $\Flags{V}$. Denoting the pullback also by $\theta(Q,\GG)$, we have that $\theta(Q,\GG)$ vanishes at $\FF$ if and only if $\Hom{I}{V}{Q}{F}{G}\neq0$.
\end{subsection}

\end{section}

\begin{section}{Two-Step Complexes}\label{TwoStep}
The proof of Theorem \ref{MainTheoremTranslated} turns on dimension counts of spaces $\Hom{I}{V}{Q}{F}{G}$. If $\MM_{\II}$ exceeds the dimension predicted, there will be a nonempty closed locus $Z$ in $(\Flags{V})^{SS}$ where $\Hom{I}{V}{Q}{F}{G}$ is nonzero for general choice of $(\FF,\GG)\in Z\times\Flags{Q}$ (note that $\Hom{I}{V}{Q}{F}{G}$ is the ``certain vector space $H$ depending on $\FF$'' of the introduction). The dimension counting techniques below will expose this as a contradiction.

To better organize our computations, we introduce two-step complexes. For flags $F_\bullet\in\textnormal{Fl}(M)$, $G_\bullet\in\textnormal{Fl}(Q)$, and a nondecreasing sequence of nonnegative integers 
\[\theta=(\theta_1\leq\theta_2\leq...\leq\theta_m\leq n-r),\]
define
\begin{equation*}
P_\theta(F_\bullet,G_\bullet):=\{\psi\in\textnormal{Hom}(M,Q)|\psi(F_a)\subseteq G_{\theta_a}\textnormal{ for }a=1,...,m\}.
\end{equation*}
It is clear that $\dim P_\theta=\sum_{a=1}^m\theta_a$. If $\FF$ is an $s$-tuple of flags on $M$, $\GG$ likewise on $Q$, and $\vartheta$ an $s$-tuple of nondecreasing sequences as above, we define the two-step complex:
\begin{equation}
\AAA(M,Q,\FF,\GG,\vartheta):=(0\rightarrow\textnormal{Hom}(M,Q)\xrightarrow{\gamma}\oplus_{p=1}^s\frac{\textnormal{Hom}(M,Q)}{P_{\theta^p}(F^p_\bullet,G^p_\bullet)}\rightarrow 0)
\end{equation}
We denote
\begin{equation*}
H^0(\AAA(M,Q,\FF,\GG,\vartheta))=\ker(\gamma)\textnormal{\phantom{xx}and\phantom{xx}}H^1(\AAA(M,Q,\FF,\GG,\vartheta))=\textnormal{coker}(\gamma),
\end{equation*}
with the corresponding lowercase $h^0$ and $h^1$ for the dimensions of these, as usual. We also define the Euler characteristic $\chi=h^0-h^1$, which is the same as the difference of dimensions between the domain and codomain of $\gamma$.  One easily computes:
\begin{equation*}
\chi(\AA(M,Q,\FF,\GG,\vartheta))=m(n-r)-\sum_{p\in S}\sum_{a=1}^{m}(n-r-\theta^p_a).
\end{equation*}
For an $s$-tuple of index sets $\HH$, let us define $s$-tuples of nondecreasing sequences $\vartheta(\HH)$ by the prescription $\theta^p_a=H^p_a-a$. In this case,
\begin{equation}\label{HomAsH0}
H^0(\AAA(M,Q,\FF,\GG,\vartheta(\HH)))=\Hom{H}{M}{Q}{F}{G}.
\end{equation}

Let $R$ be an $e$-dimensional subspace of $M$ in Schubert position $\EE$ with respect to the flags $\FF$. The natural restriction map $\rho:\textnormal{Hom}(M,Q)\rightarrow\textnormal{Hom}(R,Q)$ is such that:
\[\rho(P_{\theta(H^p)}(F^p_\bullet,G^p_\bullet))=P_{\theta(Y^p)}(F^p(R)_\bullet,G^p_\bullet),\]
where $\YY\in\binom{[n-r+e]}{e}^s$ is given by 
\begin{equation}\label{IndexSetY}
Y^p_a=H^p_{E^p_a}-E^p_a+a\textnormal{ for } a=1,...,e,\textnormal{ } p=1,...,s.
\end{equation}
If $\CC=\AAA(M,Q,\FF,\GG,\vartheta(\HH))$ and $\CC(R)=\AAA(R,Q,\FF(R),\GG,\vartheta(\YY))$, the above shows that $\rho$ induces a surjective map of complexes $\rho:\CC\twoheadrightarrow\CC(R)$. From this, we obtain the following useful proposition. It says roughly that if $\FF$ is arbitrary, and $\GG$ is chosen in general position with respect to $\FF$, then the dimension of $\Hom{H}{M}{Q}{F}{G}$ (the quantity of interest) is controlled by the dimension of $\textnormal{Hom}_\YY(R,Q,\FF(R),\GG)$, where $R$ is a certain subspace of $M$.

\begin{prop}\label{H1Prop}
Fix $\FF\in\Flags{M}$. Let $\mathfrak{O}=\mathfrak{O}_{\FF,\HH}$ be an open subset of $\Flags{Q}$ such that $\textnormal{hd}_\HH(\FF,\GG)=(D,e,\EE)$ is constant over $\GG\in\mathfrak{O}$ and such that the morphism
\begin{equation}
\textnormal{\textbf{H}}_{\HH,\EE}|_\mathfrak{O}=\textnormal{\textbf{H}}_{\HH,\EE}\times_{\Flags{M}\times\Flags{Q}}(\{\FF\}\times\mathfrak{O})\rightarrow\{\FF\}\times\mathfrak{O}
\end{equation}
is flat and surjective. A nonempty such $\mathfrak{O}$ exists by Lemma \ref{GeneralConfiguration} and generic flatness. If $\GG$ is in $\mathfrak{O}$ and $\psi$ is a general element of $\textnormal{Hom}_\HH(M,Q,\FF,\GG)$ with kernel $R$, then the restriction $\rho$ induces an isomorphism $H^1(\CC)\tilde{\rightarrow} H^1(\CC(R))$.
\end{prop}
\begin{proof}
Recall the parameter spaces of Section \ref{Parameter Spaces}. We have by Lemma \ref{ModuliDimension} that $\UnivHom{H}{E}$ is surjective and smooth over $\UnivInt{E}$ of relative dimension:
\begin{equation}\label{H1PropOne}
\textnormal{rel dim}(\UnivHom{H}{E}\rightarrow\UnivInt{E})=\dim\Flags{Q}+\chi(\CC)-\chi(\CC(R)).
\end{equation}
Under the natural projection $\UnivHom{H}{E}\rightarrow\Flags{M}\times\Flags{Q}$, the fiber over a point $(\FF,\GG')$ of $\{\FF\}\times\mathfrak{O}$ is a dense open subset (by choice of $\EE$) of $\textnormal{Hom}_\HH(M,Q,\FF,\GG')$. In particular, the fibers are irreducible. Since flat maps are open, it follows that $\UnivHom{H}{E}|_\mathfrak{O}$ itself is irreducible and that the fiber over $(\FF,\GG)$ has dimension $\dim\UnivHom{H}{E}|_\mathfrak{O}-\dim\Flags{Q}$. But we know from (\ref{HomAsH0}) that the dimension of this fiber equals $h^0(\CC)$, so:
\begin{equation}\label{H1PropTwo}
h^0(\CC)=\dim\UnivHom{H}{E}|_\mathfrak{O}-\dim\Flags{Q}.
\end{equation}
Let $U_\FF$ denote the open, irreducible image of $\UnivHom{H}{E}|_\mathfrak{O}$ in
\begin{equation}\label{H1PropTwoPointFive}
\cell_\EE(\FF)=\textbf{U}_\EE\times_{\Flags{V}}\{\FF\}.
\end{equation}
By smoothness of $\UnivHom{H}{E}$ over $\UnivInt{E}$, the scheme $\UnivHom{H}{E}\times_{\UnivInt{E}}U_\FF$ is equidimensional and smooth over $U_\FF$ of relative dimension given by (\ref{H1PropOne}). Since $\UnivHom{H}{E}|_\mathfrak{O}$ is an open subset of $\UnivHom{H}{E}\times_{\UnivInt{E}}U_\FF$, we have
\begin{equation}\label{H1PropThree}
\dim\UnivHom{H}{E}|_\mathfrak{O}=\dim U_\FF+\dim\Flags{Q}+\chi(\CC)-\chi(\CC(R)).
\end{equation}
By the description of $U_\FF$ as an open subset of the Schubert intersection (\ref{H1PropTwoPointFive}), we can combine (\ref{H1PropTwo}) and (\ref{H1PropThree}) to obtain:
\begin{equation}\label{H1PropFour}
h^0(\CC)\leq\dim(\cell_\EE(\FF)\textnormal{ at }R)+\chi(\CC)-\chi(\CC(R)),
\end{equation}
where the first summand above should be understood as the dimension of the largest irreducible component of $\cell_\EE(\FF)$ passing through $R$.

The first summand in (\ref{H1PropFour}) is bounded by the dimension of the Zariski tangent space to $\cell_\EE(\FF)$ at $R$, which is given by
\[\textnormal{Hom}_\EE(R,M/R,\FF(R),\FF(M/R)).\]
The chosen map $\psi:M\rightarrow Q$ induces an injection:
\begin{equation}\label{H1PropFive}
\textnormal{Hom}_\EE(R,M/R,\FF(R),\FF(M/R))\hookrightarrow \textnormal{Hom}_\YY(R,Q,\FF(R),\GG),
\end{equation}
where $\YY$ is as in (\ref{IndexSetY}). It is easy to see that
\begin{equation}\label{H1PropSix}
\textnormal{Hom}_\YY(R,Q,\FF(R),\GG)=H^0(\CC(R)).
\end{equation}
Thus, it follows from (\ref{H1PropFour}) and (\ref{H1PropFive}) that
\[h^0(\CC)\leq h^0(\CC(R))+\chi(\CC)-\chi(\CC(R)).\]
This rearranges to the inequality $h^1(\CC)\leq h^1(\CC(R))$. But $\rho:H^1(\CC)\rightarrow H^1(\CC(R))$ is surjective by the Snake Lemma, so the proposition follows.
\end{proof}

The reader may wish to jump now to Appendix \ref{HornAppendix}. The proof there of the Horn conjecture uses Proposition \ref{H1Prop} in a straightforward way. In that regard, it serves as a nice warm-up for the more complicated argument of Section \ref{MainProof}.
\end{section}

\begin{section}{Outline of the Proof of Theorem \ref{MainTheoremTranslated}}\label{Outline}
The argument in Section \ref{MainProof} runs roughly as follows. We assume that Theorem \ref{MainTheoremTranslated} is false, i.e. that $\MM_\II$ has dimension at least 2. Then, $\LL_\II$ has a base locus, which we can lift by $\pi$ to $\Flags{V}$. Take an irreducible component $Z$ which meets the open, semistable part of $\Flags{V}$. Then, for general $(\FF,\GG)$ in $Z\times\Flags{Q}$, it is easy to see that $\textnormal{Hom}_\II(V,Q,\FF,\GG)\neq0$ (Lemma \ref{Contradiction}). This will end up being contradicted.

Using the hypothesis $H^0(\LL_\II)=2$, we find (Proposition \ref{DominantMap}) that $Z$ is dominated by $\UnivPairedHom{I}{K}{J}$ for a certain choice of $\KK,\JJ$. Choose a general point $\FF$ in $Z$ and a general point $(\FF,\GG,\GG',\phi,\phi')$ in $\UnivPairedHom{I}{K}{J}$ over $\FF$. Set $S=\ker\phi$, $S'=\ker\phi'$, $T=S\cap S'$. 

The goal is to show that $\Hom{I}{V}{Q}{F}{G}=0$. The quantity $\textnormal{Hom}_\II$ is computable when the flags are generic (cf. \cite{Bel06}). Since $\FF$ is in $Z$, genericity cannot be assumed. However, Proposition \ref{H1Prop} gives a link between $\textnormal{Hom}$ for $V$ with flags $\FF$ and $\GG$ to $\textnormal{Hom}$ for $S$ with flags $\FF(S)$ and $\GG$. It is more convenient to express the link in terms of $H^1$'s, but these relate directly to the $\textnormal{Hom}$'s in that they indicate their deviations from the expected value. Specifically, Proposition \ref{H1Prop} applied to $\phi$ tells us that:
\[H^1(\AAA(V,Q,\FF,\GG,\vartheta(\II)))\cong H^1(\AAA(S,Q,\FF(S),\GG,\vartheta(\tilde{\II}))).\]
Again, if $(\FF(S),\GG)$ was general, the right hand side would be computable, in fact, zero by Horn's conjecture. But $\FF(S)$ is required to have a subspace $T$ of $S$ in a certain Schubert position, so it cannot be assumed generic. However, a variant of Horn's conjecture is proven in Section \ref{DescendToS}, for which $(\FF(S),
\GG)$ is general enough.

We conclude that the right hand side of the above equation is $0$. Hence, the left hand side is $0$ and $\Hom{I}{V}{Q}{F}{G}$ has its expected dimension. A consequence of the codimension condition on $\II$ is that the expected dimension is $0$. Contradiction.

\end{section}

\begin{section}{Proof of Theorem \ref{MainTheoremTranslated}}\label{MainProof}
Assume $\II$ satisfies the hypotheses of Theorem \ref{MainTheoremTranslated} and suppose to the contrary that $\dim\MM_{\II}\geq2$. Since $\LL_\II$ is ample and $\MM_{\II}$ is projective and normal, it follows that there is a nonempty base locus $Z'$ of $\MM_{\II}$ where all sections of $\LL_\II$ vanish. Fix once and for all an irreducible component of the preimage of $Z'$ in $(\Flags{V})^{SS}$ and take its closure $Z$ in $\Flags{V}$. Since $Z$ contains semistable points, there is a Zariski open subset $U_Z$ such that for all $\FF$ in $U_Z$, the parabolic vector space $(V,\FF,n-r,\lambda(\II))$ is semistable. Let $U_{\Flags{Q}}$ be the open subset of $\Flags{Q}$ such that $\theta(Q,\GG)$ (see Section \ref{ThetaSection}) is not the zero section of $\LL_\II$ for any $\GG$ in $U_{\Flags{Q}}$. We establish the statement to be contradicted.

\begin{lem}\label{Contradiction}
If $(\FF,\GG)\in U_Z\times U_{\Flags{Q}}$, the vector space $\Hom{I}{V}{Q}{F}{G}$ is nonzero.
\end{lem}
\begin{proof}
The divisor of $\Flags{V}$ associated to $\theta(Q,\GG)$ is
\[D_\GG=\{\FF'\in\Flags{V}|\Hom{I}{V}{Q}{F'}{G}\neq0\}\]
Since $\FF\in U_Z$ maps into the base locus of $\LL_\II$, every such divisor passes through $\FF$.
\end{proof}

By Lemma \ref{GeneralConfiguration} and generic flatness, there is a largest nonempty open subset $\bar{U}$ of $Z\times\Flags{Q}\times\Flags{Q}$ over which the $\textnormal{Hom}'_\II$ data is constant, say $\textnormal{hd}'_\II(\FF,\GG,\GG')=(D,f,g,\KK,\JJ)$ for all $(\FF,\GG,\GG')$ in $\bar{U}$, and such that $\UnivPairedHom{I}{K}{J}|_{\bar{U}}\rightarrow\bar{U}$ is flat and surjective. Note by Lemma \ref{Contradiction}, we have strict inequality $f<r$. 

\begin{prop}\label{DominantMap}
The morphism
\begin{equation}\label{MapToFlags}
\textnormal{\textbf{H}}_{\II,\KK,\JJ}\rightarrow \Flags{V}\times\Flags{Q}\times\Flags{Q}
\end{equation}
factors through a dominant map $\textnormal{pr}$ to $Z\times\Flags{Q}\times\Flags{Q}$.
\end{prop}
\begin{proof}
Clearly by construction the image of the map (\ref{MapToFlags}) contains a dense open subset of $Z\times\Flags{Q}\times\Flags{Q}$ (namely $\bar{U}$). To prove the proposition, it suffices to show that the image of the projection $\textnormal{pr}_1:\UnivPairedHom{I}{K}{J}\rightarrow\Flags{V}$ 
lies in $Z$.

Let $\BB$ denote the image of $\textnormal{pr}$ in $\Flags{V}\times\Flags{Q}\times\Flags{Q}$. Since in particular, $\bar{U}\subset\BB$, we have that $\BB$ dominates $\Flags{Q}\times\Flags{Q}$. Thus, the general element $(\FF,\GG,\GG')\in\BB$ is such that $\theta(Q,\GG)$ and $\theta(Q,\GG')$ form a basis for $H^0(\LL_\II)$. If $(\phi,\phi')$ is any point of $\UnivPairedHom{I}{K}{J}$ over $(\FF,\GG,\GG')$, then $\textnormal{rk}\phi=\textnormal{rk}\phi'=r-f>0$. It follows that $\theta(Q,\GG)$, $\theta(Q,\GG')$ vanish at $\FF$. So $\FF$ is in some component of the inverse image of $Z'$ in $\Flags{V}$. In sum we have:
\begin{itemize}
\item The image of $\textnormal{pr}_1$  contains a dense open subset of $Z$.
\item The image of $\textnormal{pr}_1$ lies in the inverse image of $Z'$ in $\Flags{V}$.
\item The image of $\textnormal{pr}_1$ is irreducible (since $\UnivPairedHom{I}{K}{J}$ is). 
\end{itemize} 
It follows that $\textnormal{pr}_1(\UnivPairedHom{I}{K}{J})\subseteq Z$.
\end{proof} 

Since $\textnormal{pr}$ is dominant (\ref{DominantMap}), we may let $W\subseteq Z\times\Flags{Q}\times\Flags{Q}$ be a nonempty open subset such that
\begin{enumerate}\label{PropertiesOfW}
\item $\UnivPairedHom{I}{K}{J}$ is flat and surjective over $W$.
\item If $(\FF,\GG,\GG')\in W$, then $\FF$ is semistable with respect to $\II$.
\end{enumerate}
Pick a general point $(\FF,\GG,\GG',\phi,\phi')$ in $\mathbf{H}_{\II,\KK,\JJ}$ whose image lies in $W$, with $S=\ker\phi$, $S'=\ker\phi'$, and $T=S\cap S'$. Then, by property (1) of $W$ and Proposition \ref{H1Prop}, we have:
\begin{equation}\label{H1VEqualsH1S}
H^1(\AAA(V,Q,\FF,\GG,\vartheta(\II)))\cong H^1(\AAA(S,Q,\FF(S),\GG,\vartheta(\tilde{\II}))),
\end{equation}
where $\tilde{\II}$ is given by $\tilde{I}^p_a=I^p_{K^p_a}-K^p_a+a$. By Proposition \ref{InducedGenericAsPossible}, the induced flags $(\FF(S),\GG)$ are ``general enough" for application of Proposition \ref{NewHorn}. Moreover, the inequalities of Proposition \ref{NewHorn} follow from property (2) of $W$. Therefore, the right hand side of (\ref{H1VEqualsH1S}) is zero. We conclude that $\textnormal{Hom}_\II(V,Q,\FF,\GG)$ has its expected dimension, \textit{which is zero}, since $\II$ is assumed to satisfy the codimension condition (\ref{CodimCondition}). So by \ref{Contradiction}, we have arrived at a contradiction. QED.
\end{section}

\begin{section}{Vanishing of $H^1$ for $S$}\label{DescendToS}

In this section, we will prove the propositions needed in the last part of Section \ref{MainProof} to show that $H^1(\AAA(S,Q,\FF(S),\GG,\vartheta(\tilde{\II})))=0$. To do this, we need to know that if $(\FF,\GG,\GG',\phi,\phi')$ is general in $\UnivPairedHom{I}{K}{J}$, the flags $(\FF(\ker\phi),\GG)\in\Flags{\ker\phi}\times\Flags{Q}$ are ``as general as possible'' with respect to one another, given the condition that the induced flags on $\ker\phi$ will always have subspace of $\ker\phi$ in Schubert position $\NN$ (namely $\ker\phi\cap\ker\phi'$). To this end, fix some $S\in\textnormal{Gr}(f,V)$. Define $Z_{S,\NN}\subseteq\Flags{S}$ to be the set of flags $\FF_S$ such that $\cell_\NN(\FF_S)$ is nonempty. Then, $Z_{S,\NN}$ is locally closed and irreducible, as it is the image of $\UnivInt{N}(S)$ in $\Flags{S}$. The next proposition proves the statement about induced flags being ``as general as possible.'' See Section \ref{Parameter Spaces} for definitions. Note that a vertical bar and a subscript following a parameter space - for example, ``$\UnivPairedHom{I}{K}{J}|_S$'' - denotes the fiber over the last subscript, as usual.

\begin{prop}\label{InducedGenericAsPossible}
The map $\textnormal{\textbf{H}}_{\II,\KK,\JJ}|_S\rightarrow Z_{S,\NN}\times\Flags{Q}:(\FF,\GG,\GG',\phi,\phi')\mapsto(\FF(S),\GG)$ is dominant. 
\end{prop}
\begin{proof}
We have a map $\UnivPairedHom{I}{K}{J}|_S\rightarrow(\UnivHom{I}{K}(V)|_S)|_Z\times\textnormal{Hom}(V,Q)$ which sends the point $(\FF,\GG,\GG',\phi,\phi')$ to $(\FF,\GG,\phi,\phi')$ (Proposition \ref{DominantMap} guarantees that $\FF\in Z$).  Let $\UU$ denote the image. Now by choice of $\KK$, the map $\UU\rightarrow\Flags{Q}$ is dominant. Fix $\GG$ in the image. To prove the proposition, it suffices to show that $\UU|_{\GG}\rightarrow Z_{S,\NN}$ is dominant. 

Suppose $(\FF,\GG,\phi,\phi')\in\UU|_{\GG}$ is a point, with say $S'=\ker\phi'$, $T=S\cap S'$. Let $G_{S,S',T}$ be the largest subgroup of $\textnormal{GL}(V)$ which acts on $S$, $S'$, and $T$, and acts trivially on $V/S$. Suppose $\vec{g}\in G^{\times s}_{S,S',T}$. Then we observe that $\phi\in\textnormal{Hom}_{\II}(V,Q,\vec{g}\FF,\GG)$ since $G$ acts trivially on $V/S$, and that $S,S'\in\cell_\KK(\vec{g}\FF)$, $T\in\cell_\JJ(\vec{g}\FF)$. Finally, we note that given a homomorphism $\phi'$ with $\ker\phi'=S'\in\cell_\KK(\vec{g}\FF)$, one can construct $\GG'\in\Flags{Q}$ so that $\phi'\in\textnormal{Hom}_{\II}(V,Q,\vec{g}\FF,\GG')$ (this is how the moduli space $\UnivHom{I}{K}$ is built over $\UnivInt{K}$ and similarly $\UnivPairedHom{I}{K}{J}$ over $\UnivPairedInt{K}{J}$). We conclude that $(\vec{g}\FF,\GG,\phi,\phi')\in\UU|_\GG$, which is to say that $G^{\times s}_{S,S',T}$ acts on $\UU|_\GG$, and this action restricts to $(\UU|_\GG)|_{T}$. 

We remark that $(\UU|_{\GG})|_T$ is \textit{nonempty} for any $g$-dimensional subspace $T$ of $S$. One can see this by using the action on $\UU|_\GG$ of the stabilizer group $H_S$ of $S$ in $\textnormal{GL}(V)$ given by $h\cdot(\FF,\GG,\phi,\phi'):=(\{hF^p_a\}_{p=1}^s,\GG,\phi\circ h^{-1},\phi'\circ h^{-1})$. This takes a subspace $T\in\cell_{\JJ}(\FF)$ to $T'=hT\in\cell_\JJ(h\FF)$, and any such $T'$ can be realized as $hT$ for suitable $h\in H$. Thus, we have reduced the problem to showing $(\UU|_{\GG})|_T\rightarrow Z_{S,T,\NN}$ is dominant for all $T$ (see below the proof for the definition of $Z_{S,T,\NN}$).

Note that $G^{\times s}_{S,S',T}$ acts by restriction to a subgroup of  $\textnormal{GL}(S)^{\times s}$ on $Z_{S,T,\NN}$. As can be seen with a basis argument, this action is transitive, and the map $(\UU|_{\GG})|_T\rightarrow Z_{S,T,\NN}$ is equivariant with respect to the above actions. The desired dominance follows.
\end{proof}

For an $f$-dimensional vector space $S$, a $g$-dimensional subspace $T$, and a Schubert position $\NN\in\binom{[f]}{g}^s$, one has an irreducible, locally closed subvariety of $\Flags{S}$ given by:
\begin{equation*}
Z_{S,T,\NN}=\{\FF_S\in\Flags{S}|T\in\Omega^{\circ}_{\NN}(\FF_S)\}.
\end{equation*}
If $g=0$, define $Z_{S,T,\NN}=\Flags{S}$. The next proposition and its proof are variants on the formulation of the Horn conjecture and its proof that appears in Appendix \ref{HornAppendix}.

\begin{prop}\label{NewHorn}
Let $S$ be an $f$-dimensional vector space with $f\leq r$, $T$ a $g$-dimensional subspace, $Q$ an $(n-r)$-dimensional vector space, $\NN\in\binom{[f]}{g}^s$ a Schubert position in $S$, and $(\FF_S,\GG)$ a general element of $Z_{S,T,\NN}\times\Flags{Q}$ (see Lemma \ref{GeneralConfiguration}). Let $\tilde{\II}\in\binom{[n-r+f]}{f}^s$. Suppose for every nonzero subspace $R$ of $S$, one has the inequality:
\begin{equation*}
\sum_{p\in S}\sum_{a=1}^{\dim R}(n-r+X^p_a-\tilde{I}^p_{X^p_a})\leq\dim(R)(n-r),
\end{equation*}
where $\XX$ is the Schubert position of $R$ in $S$ with respect to $\FF_S$. Then the vector space $H^1(\AAA(S,Q,\FF_S,\GG,\vartheta(\tilde{\II})))$ is zero.
\end{prop}
\begin{proof}
We proceed by induction on the dimension $f$ of $S$. If $f=1$, then $Z_{S,T,\NN}=\textnormal{pt}$, and $\GG$ is a general element of $\Flags{Q}$. In this case, $H^0(\AAA(S,Q,\FF_S,\GG,\vartheta(\tilde{\II})))$ is the space of all $\phi:S\rightarrow Q$ such that $\textnormal{Im}(\phi)$ is contained in $\bigcap_{p\in S}G_{\tilde{I}^p-1}$. Since $\GG$ is generic, we compute the dimension $h^0$ of this space to be 
\begin{equation*}
(n-r)-\sum_{p\in S}(n-r+1-\tilde{I}^p),
\end{equation*}
a nonnegative number by the inequality hypothesis. This number is also equal to $\chi(\AAA(S,Q,\FF_S,\GG,\vartheta(\tilde{\II})))$, so $h^1=0$, as desired.

Assume now that $f\geq2$. Let $\phi\in\textnormal{Hom}_{\tilde{\II}}(S,Q,\FF_S,\GG)$ be a general element. We observe that the inequality for $R=S$ can be expressed as $\chi(\AAA(S,Q,\FF_S,\GG,\vartheta(\tilde{\II})))\geq 0$. If the general element $\phi$ is zero, then $h^1=0$ follows, so we may as well assume $\phi\neq 0$. Let $\tilde{f}<f$ be the dimension of $\tilde{S}=\ker\phi$ and set $\tilde{T}=\tilde{S}\cap T$. Let $\YY$ be the Schubert position of $\tilde{S}$ in $S$ with respect to $\FF_S$, and let $\tilde{\NN}$ be the Schubert position of $\tilde{T}$ in $\tilde{S}$. By the genericity hypothesis on the flags and Proposition \ref{H1Prop}, we have 
\begin{equation}\label{NewHornEqn1}
H^1(\AAA(S,Q,\FF_S,\GG,\vartheta(\tilde{\II})))=H^1(\AAA(\tilde{S},Q,\FF_S(\tilde{S}),\GG,\vartheta({\tilde{\II}'}))),
\end{equation}
where $\tilde{\II}'\in\binom{[n-r+\tilde{f}]}{\tilde{f}}^s$ is defined by $\tilde{I}'^p_a=(\tilde{I}^p_{Y^p_a}-Y^p_a+a)$ for $a=1,...,\tilde{f}$.

Let $\tilde{G}$ be the subgroup of $\textnormal{GL}(S)$ consisting of those group elements which act on $T$, $\tilde{S}$, and $\tilde{T}$, and act trivially on $S/\tilde{S}$. If $\vec{g}\in\tilde{G}^{\times s}$, then it is easy to see that $(\vec{g}\FF_S,\GG)\in Z_{S,T,\NN}\times\Flags{Q}$ and $\phi\in\textnormal{Hom}_{\tilde{\II}}(S,Q,\vec{g}\FF_S,\GG)$. On the other hand, a straightforward argument with bases shows that $\tilde{G}^{\times s}$ acts transitively on the set of flags $Z_{\tilde{S},\tilde{T},\tilde{\NN}}$. Thus, the general pair $(\FF_S,\GG)\in Z_{S,T,\NN}\times\Flags{Q}$ induces a general pair $(\FF_S(\tilde{S}),\GG)\in Z_{\tilde{S},\tilde{T},\tilde{\NN}}$ (compare with the proof of Proposition \ref{InducedGenericAsPossible}).

We are now in position to apply the inductive hypothesis to the right hand side of (\ref{NewHornEqn1}). We need only check that the appropriate inequalities hold for all subspaces $\tilde{R}$ of $\tilde{S}$. But this follows immediately from regarding $\tilde{R}$ as a subspace of $S$, where inequalities are assumed to hold, and the fact that $\tilde{R}\subseteq\tilde{S}$ has Schubert position $\tilde{\XX}$ in $\tilde{S}$ with respect to $\FF_S(\tilde{S})$ if and only if $\tilde{R}$ has Schubert position $\YY_{\tilde{\XX}}$ in $S$ with respect to $\FF_S$.  
\end{proof}

\end{section}

\begin{appendix}

\begin{section}{Representability}\label{Representability}
The goal of this appendix is to prove Lemma \ref{ModuliDimension}. Points (1) and (2) are proven in \cite{Bel06} and \cite{Bel07}. We will prove (3), (4), and (5). To begin, we present a lemma from the appendix of \cite{Bel06}. The dimension count is new, but follows from Belkale's argument in a straightforward way. 

\begin{lem}\label{RepLemma}
Let $\PP$ be a vector bundle of rank $\rho$ on a scheme $Z$. Assume that $\PP$ has a subbundle $\Sigma$ which is filtered by subbundles $0=\Sigma_0\subseteq\Sigma_1\subseteq...\subseteq\Sigma_k=\Sigma$. Let $0=\ell_0\leq\ell_1\leq...\leq\ell_k\leq\rho$ be nonnegative integers. Let $\alpha:\textnormal{Sch}/Z\rightarrow\textnormal{Set}$ be the contravariant functor which associates to a scheme $T$ over $Z$ the set of complete filtrations $\FF_\bullet$ of $\PP|_{T}$ by subbundles such that the composite map \begin{equation*}
\Sigma_a|_{T}\rightarrow\Sigma|_{T}\rightarrow\PP|_{T}/\FF_{\ell_a}
\end{equation*}
is $0$ for all $a=0,...,k$. Then $\alpha$ is representable by a scheme $\textnormal{\textbf{A}}$ which is surjective and smooth over $Z$ of relative dimension
\begin{equation}\label{GeneralRelDim}
\dim\textnormal{Fl}({k^\rho})-\sum_{t=1}^{\rho-1}\textnormal{rk}(\Sigma_{c_t}),
\end{equation}
where $c_t=\textnormal{max}\{a|\ell_a\leq t\}$. Moreover, if $Z$ is irreducible, so is $\textnormal{\textbf{A}}$.

Also, if $\tilde{\alpha}$ is the same functor as $\alpha$ with the stricter condition that
\begin{equation*}
\Sigma_a|_{T}\rightarrow\Sigma|_{T}\rightarrow\PP|_{T}/\FF_{t}
\end{equation*}
is exact whenever $\ell_a\leq t \leq \ell_{a+1}-1$, then $\tilde{\alpha}$ is represented by an open, possibly empty subscheme $\tilde{\textnormal{\textbf{A}}}$  of $\textnormal{\textbf{A}}$, which is also surjective and smooth over $Z$, and irreducible if $Z$ is.
\end{lem}

\begin{rem}\label{BaseChangeRemark}
The schemes $\textnormal{\textbf{A}}$ base change properly, in the sense that if $Z'\rightarrow Z$ is a morphism, then $\textnormal{\textbf{A}}\times_Z Z'$ represents the functor corresponding to the pulled back bundle $\PP|_{Z'}$ with the pulled back filtration. In particular, if $Z'\rightarrow Z$ is a morphism of irreducible schemes, then $\tilde{\textnormal{\textbf{A}}}|_{Z'}$ is irreducible.
\end{rem}
\begin{rem}\label{DimensionRemark}
In the case where $0=\Sigma_0\subseteq\Sigma_1\subseteq...\subseteq\Sigma_k=\Sigma$ is a \textit{complete} filtration of $\Sigma$ (with all inclusions proper) and $0=\ell_0<\ell_1<...<\ell_k\leq\rho$ (with all inequalities strict, hence corresponding to some $L\in\binom{\rho}{k}$), the quantity $\sum_{t=1}^{\rho-1}\textnormal{rk}(\Sigma_{c_t})$ in (\ref{GeneralRelDim}) is equal to $k\rho-\sum_{a=1}^k\ell_a$. This may be rewritten as $|\lambda(L)|+\dim\textnormal{Fl}(k^k)$. Thus, the relative dimension of $\textnormal{\textbf{A}}$ over $Z$ is then:
\begin{equation}\label{SpecificRelDim}
\dim\textnormal{Fl}(k^{\rho})-|\lambda(L)|-\dim\textnormal{Fl}(k^k)
\end{equation}
(where the lower ``$k$'' of course refers to the field).
\end{rem}

Let $\KK$, $\JJ$, and $\NN$ be as in Section \ref{Parameter Spaces} (\ref{IndexSets}). We will build the scheme $\UnivPairedInt{K}{J}(V)$ in several steps. Recall that we would like $\UnivPairedInt{K}{J}$ on the level of points to be the set of all tuples $(S,S',T,\FF)$, where $S$ and $S'$ are $f$-dimensional subspaces of $V$ with $g$-dimensional intersection $T$, satisfying $S,S'\in\cell_\KK(\FF)$, $T\in\cell_\JJ(\FF)$. Thus, a natural starting point would be to consider the contravariant functor $\AAA_{f,f,g}:\textnormal{Sch}/k\rightarrow\textnormal{Set}$ which associates to each $k$-scheme $Y$ the set of triples consisting of two rank $f$ subbundles $\SSS$ and $\SSS'$ of $V\otimes\OO_Y$, and a rank $g$ subbundle $\TT$ of $V\otimes\OO_Y$ such that $\TT\rightarrow V\otimes\OO_Y$ is precisely the kernel of
\[V\otimes\OO_Y\rightarrow(V\otimes\OO_Y/\SSS)\oplus(V\otimes\OO_Y/\SSS').\]
The functor $\AAA_{f,f,g}$ is representable by a smooth, irreducible scheme $A_{f,f,g}$ of dimension (if nonempty) $2(f-g)(r-f)+g(r-g)$. 

Indeed, one builds $A_{f,f,g}$ by starting with the Grassmannian of $g$-dimensional subspaces of $V$, which has dimension $g(r-g)$. Using the tautological bundle on the Grassmannian, one can build a scheme over $\textnormal{Gr}(g,V)$ (smooth of relative dimension $(f-g)(r-f)$) whose fiber over $T$ is the set of all $f$-dimensional $S$ in $V$ containing $T$. Similarly, build over that scheme the scheme whose fiber over $(T,S)$ is the set of all $S'$ containing $T$. The locus where $T$ is precisely the intersection of $S$ and $S'$ is open. This proves (3) in Lemma \ref{ModuliDimension}.

We now define a functor $\BB:\textnormal{Sch}/A_{f,f,g}\rightarrow\textnormal{Set}$ which associates to $Y$ the same data as $\AAA_{f,f,g}$ with the additional data of $s$-many complete filtrations by subbundles $\{\FF_{\bullet}^{\TT,p}\}_{p=1}^s$ of $\TT$ (so $k$-points of $\BB$ look like $(S,S',T,\FF_T)$). This is clearly representable by the flag bundle
\[B=\textnormal{Fl}(\TT)\times_{A_{f,f,g}}..._{s\textnormal{-many}}...\times_{A_{f,f,g}}\textnormal{Fl}(\TT)\]
over $A_{f,f,g}$, which is irreducible, surjective, and smooth over $A_{f,f,g}$ of relative dimension $\dim\Flags{k^g}$.

Similarly, define a functor $\CC:\textnormal{Sch}/B\rightarrow\textnormal{Set}$ which associates to $Y$ the same data as $\BB$ with the additional data of $s$-many complete filtrations by subbundles $\{\FF_{\bullet}^{\SSS,p}\}_{p=1}^s$ of $\SSS$, subject to the constraints that for $p=1,...,s,a=0,...,g$, the kernel of $\TT\rightarrow\SSS/\FF_t^{\SSS,p}$ is precisely $\FF_a^{\TT,p}$ whenever $N^p_a\leq t<N^p_{a+1}$ - here we take $N^p_0=0, N^p_{g+1}=f+1$. The $k$-points of $\CC$ look like $(S,S',T,\FF_T,\FF_S)$ so that $\FF_\TT$ is the $s$-tuple of flags induced on $T$ by $\FF_S$ and $T$ is in Schubert position $\NN$ in $S$ with respect to $\FF_S$. 

The functor $\CC$ sits atop a tower of functors,
\[\CC=\CC_s\rightarrow\CC_{s-1}\rightarrow...\rightarrow\CC_2\rightarrow\CC_1\rightarrow\BB,\]
where $\CC_1$ associates to $Y$ the same data as $\BB$ with the additional data of a complete filtration by subbundles $\FF_\bullet^{\SSS,1}$ of $\SSS$, subject to the constraints that for $a=1,...,g$, the kernel of $\TT\rightarrow\SSS/\FF_t^{\SSS,1}$ is precisely $\FF_a^{\TT,1}$ whenever $N^1_a\leq t<N^1_{a+1}$. Apply Lemma \ref{RepLemma} and the subsequent remarks with $Z=B$, $\rho=f$, $k=g$, $\PP=\SSS$, $\Sigma=\TT$ with its universal complete filtration $\FF^{\TT,1}_{\bullet}$ on $B$, and $L=N^1$, to obtain a representing scheme $C_1$ which is irreducible, smooth, and surjective over $B$ of relative dimension $\dim\textnormal{Fl}(k^f)-|\lambda(N^1)|-\dim\textnormal{Fl}(k^g)$. Repeat this $s$-many times to obtain at the top of the tower a representing scheme $C$ for $\CC$. Then define a functor $\CC':\textnormal{Sch}/C\rightarrow\textnormal{Set}$ in the obvious way, so that its $k$-points are $(S,S',T,\FF_T,\FF_S,\FF_{S'})$ with the same Schubert conditions also for $\FF_{S'}$. Repeating the argument of $\CC$ for $\CC'$, one obtains a representing scheme $C'$ which is irreducible, smooth, and surjective over $B$ of relative dimension
\[2\dim\Flags{k^f}-2\dim\Flags{k^g}-2\sum_{p=1}^{s}|\lambda(N^p)|.\]

Finally, define a functor $\DD:\textnormal{Sch}/C'\rightarrow\textnormal{Set}$ which associates to $Y$ the same data as $\CC'$ with the additional data of $s$-many complete filtrations $\{\FF^p_\bullet\}_{p=1}^s$ of $V\otimes\OO_Y$, subject to the constraints that for $b=0,...,f$, the kernel of $\SSS\rightarrow V\otimes\OO_Y/\FF_t^p$ is precisely $\FF_b^{\SSS,p}$ and the kernel of $\SSS'\rightarrow V\otimes\OO_Y/\FF_t^p$ is precisely $\FF_b^{\SSS',p}$ whenever $K^p_b\leq t<K^p_{b+1}$. The points $\DD(k)$ correspond bijectively to the set we called $\UnivPairedInt{K}{J}$.

\begin{prop}\label{UnivInt is Representable}
Let $\DD$ be the functor defined above. Then $\DD$ is representable by a scheme which we call $\textnormal{\textbf{U}}_{\KK,\JJ}(V)$, which is smooth, surjective, and irreducible over $\textnormal{A}_{f,f,g}$ of relative dimension:
\begin{equation*}
\dim\Flags{V}+\sum_{p=1}^s(|\lambda(J^p)|-2|\lambda(K^p)|-2|\lambda(N^p)|).
\end{equation*}
This proves (4) of Lemma \ref{ModuliDimension}.
\end{prop}
\begin{proof}
It suffices to prove the assertions above for the functor $\tilde{\DD}$ which associates to $Y$ the same data as $\DD$, with the milder condition that for $b=0,...,f$, the maps 
\[\FF_b^{\SSS,p}\rightarrow V\otimes\OO_Y/\FF^p_{K^p_b} \textnormal{ and } \FF_b^{\SSS',p}\rightarrow V\otimes\OO_Y/\FF^p_{K^p_b}\]
are $0$, for then $\DD$ is an open subfunctor of $\tilde{\DD}$.

The conditions of $\CC'$ imply that for each scheme $Y$ over $C'$, $a=0,...,g$, and $p=1,...,s$, we have a canonical inclusion of bundles $\FF_a^{\TT,p}\rightarrow\FF_b^{\SSS,p}\oplus\FF_b^{\SSS',p}$ whenever $N^p_a\leq b <N^p_{a+1}$. The cokernel of the inclusion is a bundle of rank $2b-a$, which we denote $\FF_b^{\SSS,p}+\FF_b^{\SSS',p}$. We may also realize this sum as the image of the map $\FF_b^{\SSS,p}\oplus\FF_b^{\SSS',p}\rightarrow V\otimes\OO_Y$, so it is naturally a subbundle of $V\otimes\OO_Y$. The maps 
\[\FF_b^{\SSS,p}\rightarrow V\otimes\OO_Y/\FF^p_{K^p_b} \textnormal{ and } \FF_b^{\SSS',p}\rightarrow V\otimes\OO_Y/\FF^p_{K^p_b}\]
are both zero if and only if 
\[\FF_b^{\SSS,p}+\FF_b^{\SSS',p}\rightarrow V\otimes\OO_Y/\FF^p_{K^p_b}\]
is zero. For each $p=1,...,s$, we have a filtration of $\Sigma=\SSS+\SSS'$ by $f$-many subbundles $\Sigma_b=\FF_b^{\SSS,p}+\FF_b^{\SSS',p}$. Applying Lemma \ref{RepLemma} $s$-many times, once for each such filtration, we see that $\tilde{\DD}$ is representable by an irreducible scheme $\tilde{D}$, which is surjective and smooth over $C'$. A computation involving the dimension count in \ref{RepLemma} gives the relative dimension of $\tilde{D}$ over $C'$ to be:
\begin{equation*}
\dim\Flags{V}-2\dim\Flags{k^f}-2(\sum_{p=1}^s|\lambda(K^p)|)+\dim\Flags{k^g}+ \sum_{p=1}^s|\lambda(J^p)|.
\end{equation*}
Combining this with the relative dimension of $C'$ over $B$ and the relative dimension of $B$ over $A_{f,f,g}$, one obtains the proposed number. 
\end{proof}

Let $\II$ and $\LL$ be as in Section \ref{Parameter Spaces}. We would like $\UnivPairedHom{\II}{\KK}{\JJ}$ to be a scheme over $\UnivPairedInt{K}{J}$ whose fiber over $(S,S',T,\FF)$ is the set of quadruples $(\GG,\GG',\phi,\phi')$ where $\GG,\GG'\in\Flags{Q}$ and $\phi\in\Hom{I}{V}{Q}{F}{G}$, $\phi'\in\Hom{I}{V}{Q}{F}{G'}$ are such that $\ker\phi=S$, $\ker\phi'=S'$. Such a scheme exists provided the functor $\EE$ is representable, where $\EE$ associates to $Y$ over $\UnivPairedInt{K}{J}$ the same data as $\DD$ above, with the additional data of $2s$-many complete filtrations $\{\GG^p_\bullet\}_{p=1}^s$, $\{\GG'^p_\bullet\}_{p=1}^s$ by subbundles of $Q\otimes\OO_Y$ and two homomorphisms of vector bundles $\phi,\phi':V\otimes\OO_Y\rightarrow Q\otimes\OO_Y$ subject to the conditions $\ker\phi=\SSS$, $\ker\phi=\SSS'$, and the composite maps 
\begin{equation*}
\FF^p_a\rightarrow V\otimes\OO_Y\xrightarrow{\phi}Q\otimes\OO_Y\rightarrow Q\otimes\OO_Y/\GG^p_{I^p_a-a}
\end{equation*}
\begin{equation*}
\FF^p_a\rightarrow V\otimes\OO_Y\xrightarrow{\phi'}Q\otimes\OO_Y\rightarrow Q\otimes\OO_Y/\GG'^p_{I^p_a-a}
\end{equation*}
are zero for $a=1,...,r$, $p=1,...,s$.

\begin{prop}\label{UnivHom is Representable}
Let $\EE:\textnormal{Sch}/\textnormal{\textbf{U}}_{\KK,\JJ}\rightarrow\textnormal{Set}$ be as above. Then $\EE$ is represented by a scheme $\textnormal{\textbf{H}}_{\II,\KK,\JJ}$ which is irreducible, smooth, and surjective over $\textnormal{\textbf{U}}_{\KK,\JJ}$ of relative dimension
\[2(r-f)(n-r)+2\dim\Flags{Q}+2\sum_{p=1}^{s}(|\lambda(L^p)|-|\lambda(I^p)|-|\lambda(K^p)|).\]
\end{prop}
\begin{proof}
Virtually identical to \cite{Bel06} Lemma A.5.
\end{proof}

\end{section}

\begin{section}{A Short Proof of the Geometric Horn Conjecture}\label{HornAppendix}
Belkale \cite{Bel06} proved, in a precise sense, that given an element $(\FF,\GG)\in\Flags{M}\times\Flags{Q}$, which is chosen sufficiently generally, the induced flags on $\ker\phi$ and $M/\ker\phi$ are themselves mutually sufficiently general, where $\phi$ is a general element in the space $\Hom{H}{M}{Q}{F}{G}$. Here ``sufficiently general'' means roughly that the flags are general enough to perform intersection theoretic computations and to calculate the dimensions of the vector spaces $\Hom{H}{M}{Q}{F}{G}$. To prove the Horn conjecture, we will require a similar result on the mutual genericity of $(\FF(\ker\phi),\GG)$. We begin by reiterating the results of Belkale.

Given $M$ of dimension $m$, there are only finitely many $s$-tuples $\HH$ of index sets in $[m]$, so by intersecting finitely many Zariski open subsets of $\Flags{M}$, we obtain an open subset $\beta$ with the property that each $s$-tuple of flags $\FF$ in $\beta$ is such that for any $s$-tuple of index sets $\HH$ in $[m]$, each of the same cardinality, the Schubert intersection $\Omega_\HH(\FF)$ is proper and transverse at every point. An element $\FF$ of $\beta$ is said to be \textit{generic for intersection theory}.

Similarly, given $\HH$, there is by Lemma \ref{GeneralConfiguration} a nonempty, largest open subset $O(M,Q,\HH)$ of $\Flags{M}\times\Flags{Q}$ such that $\textnormal{hd}_\HH$ is constant on this open set, say $\textnormal{hd}_\HH=(D,e,\EE)$ and such that $\UnivHom{H}{E}\rightarrow\Flags{M}\times\Flags{Q}$ is flat and surjective after base change to $O(M,Q,\HH)$.

Belkale defined a simultaneous choice of open subset $A_{M,Q}\subseteq\Flags{M}\times\Flags{Q}$ for every pair of nonzero vector spaces $M$ and $Q$ such that:

\begin{itemize}
\item The choice is ``functorial for isomorphisms,'' in the sense that if $M\rightarrow M'$, $Q\rightarrow Q'$ are vector spaces isomorphisms, then the induced isomorphism of varieties $\Flags{M}\times\Flags{Q}$ to $\Flags{M'}\times\Flags{Q'}$ sends $A_{M,Q}$ isomorphically onto $A_{M',Q'}$. In particular, $A_{M,Q}$ is stable under the diagonal action of $\textnormal{GL}(M)\times\textnormal{GL}(Q)$ on $\Flags{M}\times\Flags{Q}$.
\item Fixing $M$ and $Q$, if $(\FF,\GG)\in A_{M,Q}$, then
\begin{itemize}
\item[G1.] $\FF$ and $\GG$ are generic for intersection theory in $M$ and $Q$ respectively.
\item[G2.] For any choice of $\HH\in\binom{[n-r+m]}{m}^s$, we have $(\FF,\GG)\in O(M,Q,\HH)$.
\item[G3.] For $\HH$ as in G2, there is a nonempty open subset of $\Hom{H}{M}{Q}{F}{G}$ such that for each $\phi$ in this open subset, if $0<\dim\ker\phi<m$, then $(\FF(\ker\phi),\FF(M/\ker\phi))$ lies in $A_{\ker\phi,M/\ker\phi}$.
\end{itemize}
\end{itemize}
Let $B_{M,Q}$ denote the subset of $A_{M,Q}$ with the additional property:
\begin{itemize}
\item[G4.] Same as G3, except that we require $(\FF(\ker\phi),\GG)$ to lie in $B_{\ker\phi,Q}$.
\end{itemize}
\begin{prop}\label{BMQExists}
For a fixed vector space $Q\neq0$, there is a collection of nonempty open subsets $\{B_{M,Q}\subseteq\Flags{M}\times\Flags{Q}\}_{M\neq0}$ with the properties G1 through G4.
\end{prop}
\begin{proof}
The proof resembles that of Section 8.2 in \cite{Bel06}. A slight difference occurs in the ``Nonemptiness'' section. In Belkale's proof, he shows that fixing a proper nonzero subspace $R$ of $M$, the map $\UnivHom{H}{E}|_R\rightarrow\Flags{R}\times\Flags{M/R}$ is dominant (for particularly chosen $\HH,\EE$). We must show instead that the map $\UnivHom{H}{E}|_R\rightarrow\Flags{R}\times\Flags{Q}$ is dominant. To do this, take $\GG$ in the open image of $\UnivHom{H}{E}|_R\rightarrow\Flags{Q}$, and consider the map:
\begin{equation}\label{BMQOne}
\UnivHom{H}{E}|_{(\phi,\GG)}\rightarrow\Flags{R},
\end{equation}
for some $\phi$ with $\ker\phi=R$ such that the set of $\FF$ with $(\phi,\FF,\GG)\in\UnivHom{H}{E}$ is nonempty. It suffices to prove that (\ref{BMQOne}) is dominant. This is seen by considering the actions of $G^{\times s}$ on both the domain and codomain of (\ref{BMQOne}), where $G$ is the largest subgroup of $\textnormal{GL}(M)$ which acts on $R$ and acts trivially on $M/R$. The action of $G^{\times s}$ on $\Flags{R}$ is transitive, and (\ref{BMQOne}) is equivariant, so dominance (in fact surjectivity) follows. 
\end{proof}

Following an idea sketched by Belkale in discussions, we now use Proposition \ref{H1Prop} to simplify his published proof of the Horn conjecture \cite{Bel06}. 

\begin{thm}[Geometric Horn]\label{GeometricHorn}
Let $\II\in\binom{[n]}{r}^s$ be arbitrary. The following are equivalent:
\begin{itemize}
\item[A.] The product $\omega_\II$ is nonzero in $H^*(\textnormal{Gr}(r,n))$
\item[B.] For all $0<f\leq r$ and all $\KK\in\binom{[r]}{f}^s$ such that $\omega_\KK$ is nonzero in $H^*(\textnormal{Gr}(f,r))$, one has the inequality:
\begin{equation}\label{DaggerIK}\tag{$\dagger^{\II}_{\KK}$}
\sum_{p=1}^s\sum_{a=1}^f (n-r+K^p_a-I^p_{K^p_a})-f(n-r)\leq 0.
\end{equation}
\end{itemize}
\end{thm}
\begin{proof}
(A)$\Rightarrow$(B). Assume (A). Let $\EE$ be a general complete flag in the $n$-dimensional space $W$. The intersection $\cell_\II(\EE)$ is nonempty, say $V\in\cell_\II(\EE)$. Note that the intersection of cells is dense in the intersection of Schubert varieties for $\EE$ generic (see \cite{Bel06} Proposition 1.1 for a proof). Let $\KK$ as in (B) be such that $\omega_\KK$ is nonzero. Then, the intersection of Schubert varieties $\Omega_\KK(\EE(V))\subseteq\textnormal{Gr}(f,V)$ is nonempty of dimension at least:
\begin{equation}\label{LHSStar}
\dim\textnormal{Gr}(f,V)-\sum_{p=1}^s\textnormal{codim}(\Omega_{K^p}(E^p(V)))=f(r-f)-\sum_{p=1}^s\sum_{a=1}^{f}(r-f+a-K^p_a)
\end{equation}
Under the inclusion $\textnormal{Gr}(f,V)\hookrightarrow\textnormal{Gr}(f,W)$, we have:
\begin{equation}\label{Star}
\Omega_\KK(\EE(V))\subseteq\Omega_\LL(\EE),
\end{equation}
where $L^p_a=I^p_{K^p_a}$. Since $\EE$ is generic, the intersection $\Omega_\LL(\EE)$ has by Kleiman transversality its expected dimension:
\begin{equation}\label{RHSStar}
f(n-f)-\sum_{p=1}^s\sum_{a=1}^f (n-f+a-I^p_{K^p_a}).
\end{equation}
By (\ref{Star}), the quantity in (\ref{LHSStar}) is less than or equal to the quantity in (\ref{RHSStar}). This rearranges to the inequality of (B)
\end{proof}

\begin{proof}
(B)$\Rightarrow$(A). Assume (B) and furthermore assume that we have proven for all $r'<r$ and $\tilde{\II}$ in $\binom{[n-r+r']}{r'}^s$ that $\omega_{\tilde{\II}}\neq0$ if for all $0<g\leq r'$ and all $\JJ\in\binom{[r']}{g}^s$ with $\omega_\JJ$ nonzero in $H^*(\textnormal{Gr}(g,r'))$ the inequality
\begin{equation*}
\sum_{p=1}^{s}\sum_{a=1}^{g}(n-r+J^p_a-\tilde{I}^p_{J^p_a})-g(n-r)\leq 0
\end{equation*}
holds. The base case when $r'=1$ is easy. For example, one may use the simple description of the cohomology ring $H^*(\textnormal{Gr}(1,N+1))$ as $\mathbb{Z}[h]/h^{N+1}$ where $h$ is the class of a hyperplane in $\textnormal{Gr}(1,N+1)=\mathbb{P}^N$.

In the general case, let $(\FF,\GG)\in B_{V,Q}$ (see Proposition \ref{BMQExists}), and let $S$ be the kernel of a general element of $\Hom{I}{V}{Q}{F}{G}$. Suppose that $S$ has dimension $f$ and lies in Schubert position $\KK$ with respect to $\FF$. By Proposition \ref{H1Prop}, we have
\begin{equation}\label{H1ForHorn}
h^1(\AAA(V,Q,\FF,\GG,\vartheta(\II)))=h^1(\AAA(S,Q,\FF(S),\GG,\vartheta(\tilde{\II}))),
\end{equation}
where $\tilde{I}^p_a=I^p_{K^p_a}-K^p_a+a$. An easy consequence of Proposition \ref{TangentCondition} is that (A) holds if and only if the left hand side of (\ref{H1ForHorn}) is zero. So it suffices to show the right hand side of (\ref{H1ForHorn}) is zero. Clearly this holds if $f=\dim S=0$.

Assume now that $f>0$. Let $0<g\leq f$  and $\JJ\in\binom{[f]}{g}^s$ be such that the product $\omega_\JJ$ is nonzero in $H^*(\textnormal{Gr}(g,S))$. Let $T\in\textnormal{Gr}(g,S)$ lie in the nonempty intersection $\cell_\JJ(\FF(S))$; here we use the fact that $(\FF,\GG)\in B_{V,Q}$ and hence $(\FF(S),\GG)\in B_{S,Q}$ is generic. One also has that $T\in\cell_{\KK_\JJ}(\FF)\subseteq\textnormal{Gr}(g,V)$. Since $\FF$ is generic, the product $\omega_{\KK_\JJ}$ is nonzero in $H^*(\textnormal{Gr}(g,r))$. The hypothesis (B) now gives the inequality
\begin{equation*}
0\geq\sum_{p=1}^s\sum_{a=1}^g (n-r+K^p_{J^p_a}-I^p_{K^p_{J^p_a}})-g(n-r) =\sum_{p=1}^s\sum_{a=1}^g (n-r+J^p_a-\tilde{I}^p_{J^p_a})-g(n-r)
\end{equation*}
This is precisely the inequality needed to apply the inductive hypothesis, so $\omega_{\tilde{\II}}\neq 0$ in $H^*(\textnormal{Gr}(f,n-r+f))$ By Proposition \ref{TangentCondition}, there is a maximal nonempty open locus $\OO$ in $\Flags{S}\times\Flags{Q}$ of points $(\FF',\GG')$ such that $h^1(\AAA(S,Q,\FF',\GG',\vartheta(\tilde{\II})))=0$. Note that $\OO$ contains $B_{S,Q}$ - in particular contains $(\FF(S),\GG)$ because $h^1$ is constant along $B_{S,Q}$. We conclude that right hand side of (\ref{H1ForHorn}) is $0$.
\end{proof}
\end{section}

\end{appendix}

\bibliographystyle{amsalpha}
\bibliography{C:/Users/dothetimmywalk/Desktop/Grad/UltimateBib/UltimateBib}{}
\end{document}